\documentclass[11pt]{article}

\title{An elementary alternative to ECH capacities}
\author{Michael Hutchings\footnote{Partially supported by NSF grant DMS-2005437.}}
\date{}

\usepackage{amssymb}
\usepackage{latexsym}
\usepackage{amsmath}
\usepackage{amsthm}
\usepackage{amscd}
\usepackage{color}

\newcommand{\mc}[1]{{\mathcal #1}}

\numberwithin{equation}{section}

\newtheorem{theorem}{Theorem}[section]

\newtheorem{lemma}[theorem]{Lemma}
\newtheorem{lemma-definition}[theorem]{Lemma-Definition}

\theoremstyle{definition}
\newtheorem{definition}[theorem]{Definition}

\newtheorem{remark}[theorem]{Remark}

\newtheorem{example}[theorem]{Example}

\renewcommand{\frak}{\mathfrak}

\newcommand{\C}{{\mathbb C}}

\newcommand{\R}{{\mathbb R}}

\newcommand{\Z}{{\mathbb Z}}

\newcommand{\op}{\operatorname}

\newcommand{\M}{\mc{M}}

\newcommand{\Ker}{\op{Ker}}

\newcommand{\bpm}{\begin{pmatrix}}
\newcommand{\epm}{\end{pmatrix}}

\renewcommand{\epsilon}{\varepsilon}

\begin{document}

\maketitle

\begin{abstract}
The ECH capacities are a sequence of numerical invariants of symplectic four-manifolds which give (sometimes sharp) obstructions to symplectic embeddings. These capacities are defined using embedded contact homology, and establishing their basic properties currently requires Seiberg-Witten theory.   In this note we define a new sequence of symplectic capacities in four dimensions using only basic notions of holomorphic curves. The new capacities satisfy the same basic properties as ECH capacities and agree with the ECH capacities for the main examples for which the latter have been computed, namely convex and concave toric domains. The new capacities are also useful for obstructing symplectic embeddings into closed symplectic four-manifolds. This work is inspired by a recent preprint of McDuff-Siegel giving a similar elementary alternative to symplectic capacities from rational SFT.
\end{abstract}

\tableofcontents

\setcounter{tocdepth}{2}

\section{Introduction}
\label{sec:intro}

We define a {\em symplectic capacity\/} to be a function $c$ which maps some set of symplectic manifolds (possibly noncompact, disconnected, and/or with boundary or corners) to $[0,\infty]$. We assume the following two properties:
\begin{description}
\item{(Monotonicity)}
If $(X,\omega)$ and $(X',\omega)$ are symplectic manifolds of the same dimension for which $c$ is defined, and if there exists a symplectic embedding $\varphi:(X,\omega)\to (X',\omega')$, then $c(X,\omega)\le c(X',\omega')$.
\item{(Conformality)}
If $r>0$ then $c(X,r\omega) = rc(X,\omega)$.
\end{description}
Various symplectic capacities are used to study symplectic embedding problems. In particular, symplectic capacities give obstructions to symplectic embeddings via the Monotonicity property, because under the hypotheses of this property, if $c(X,\omega) > c(X',\omega')$, then a symplectic embedding $(X,\omega)\to (X',\omega')$ cannot exist. See e.g.\ \cite{qsg} for a survey of symplectic capacities.

Perhaps the most basic example of a symplectic capacity is the {\em Gromov width\/} $c_{\op{Gr}}$.  For $a>0$, define the ball
\[
B^{2n}(a) = \{z\in\C^n\mid \pi|z|^2\le a\}
\]
with the restriction of the standard symplectic form $\sum_{i=1}^n dx_i\, dy_i$ on $\C^n=\R^{2n}$. If $\op{dim}(X)=2n$, then $c_{\op{Gr}}(X,\omega)$ is defined to be the supremum over $a$ such that there exists a symplectic embedding $B^{2n}(a)\to(X,\omega)$. The celebrated Gromov nonsqueezing theorem \cite{nonsqueezing} is equivalent to the statement that the cylinder
\[
Z^{2n}(a) = \{z\in\C^n \mid \pi|z_1|^2\le a\}
\]
has Gromov width equal to $a$.

While the Gromov width has a very simple definition, it is difficult to use by itself for studying symplectic embedding problems, since it is defined in terms of symplectic embeddings. In general, there is a gap we would like to bridge between (1) symplectic capacities with simple geometric definitions that can be hard to compute, such as the Gromov width; and (2) symplectic capacities defined using Floer-theoretic or related machinery which are more computable, but whose definition requires substantial technical work.

One example of the latter type of capacity is the sequence of Ekeland-Hofer capacities defined using variational methods in \cite{eh}, or the conjecturally equivalent capacities defined in \cite{gh} using positive $S^1$-equivariant symplectic homology.

Another example, which is the focus of the present paper, is the sequence of ECH capacities introduced in \cite{qech}; see the expositions in \cite{pnas,bn} and \S\ref{sec:compare} below. Let $(X,\omega)$ be a symplectic four-manifold, not necessarily closed or connected. The ECH capacities of $(X,\omega)$ are a sequence of real numbers
\[
0 = c_0^{\op{ECH}}(X,\omega) < c_1^{\op{ECH}}(X,\omega)\le c_2^{\op{ECH}}(X,\omega) \le \cdots \le +\infty.
\]
Monotonicity of ECH capacities means that if $(X',\omega')$ is another symplectic four-manifold, and if there exists a symplectic embedding $(X,\omega)\to(X',\omega')$, then
\begin{equation}
\label{eqn:obstruction}
c_k^{\op{ECH}}(X,\omega) \le c_k^{\op{ECH}}(X',\omega')
\end{equation}
for all $k$. This obstruction is known to be sharp in some cases. For example, McDuff \cite{mcd} showed that if $X$ and $X'$ are open ellipsoids in $\R^4$ with the restriction of the standard symplectic form, then there exists a symplectic embedding $X\to X'$ if and only if $c_k^{\op{ECH}}(X)\le c_k^{\op{ECH}}(X')$ for all $k$. More generally, Cristofaro-Gardiner \cite{concaveconvex} showed that this sharpness result extends to the case when $X$ is an open ``concave toric domain'', and $X'$ is a ``convex toric domain'', in $\R^4$; see the definitions below. The ECH capacities are defined using embedded contact homology \cite{bn}, and the proof of the symplectic embedding obstruction in \eqref{eqn:obstruction} uses cobordism maps on embedded contact homology, which currently need to be defined using Seiberg-Witten theory\footnote{Heuristically one might expect to define such a cobordism map just by counting holomorphic curves. Although this is possible in some special cases \cite{gerig,rooney}, in general there are severe transversality difficulties with multiply covered curves; see \cite[\S5.5]{bn} for explanation.}.

More recently, Siegel \cite{siegel} used rational symplectic field theory (SFT) \cite{egh} to define a set of symplectic capacities which are well suited to studying stabilized symplectic embedding problems. These capacities are not yet rigorously defined because the technical foundations of rational SFT are still a work in progress. However McDuff-Siegel \cite{ms} showed that the key applications of Siegel's capacities can be proved rigorously, using a replacement of some of Siegel's capacities by an alternate set of capacities with a more elementary definition directly in terms of holomorphic curves with local tangency constraints.

More generally, one can hope that capacities extracted from Floer theories can be understood geometrically without passing through Floer theory, or at least can be replaced by more elementary capacities with the same applications. Roughly speaking, following the idea of the McDuff-Siegel capacities, the elementary capacities that we have in mind are answers to versions of the following question: {\em What is the minimal energy for which holomorphic curves satisfying certain conditions are guaranteed to exist?\/}

The purpose of this note is to pursue this direction for the ECH capacities. Namely we give an elementary definition of a sequence of symplectic capacities for symplectic four-manifolds, which we denote by $c_k$, which are defined directly in terms of holomorphic curves constrained to pass through $k$ points. We show that the capacities $c_k$ have the same basic properties as ECH capacities and agree with them in important examples. In particular, this allows some of the applications of ECH capacities to be re-proved without using Seiberg-Witten theory. The capacities $c_k$ also give good obstructions to symplectic embeddings into some closed symplectic four-manifolds with $b_2^+=1$ such as $\C P^2$ or $S^2\times S^2$, whose ECH capacities are not known. At the end, we define an even simpler sequence of capacities $\widehat{c}_k$ in any dimension, which conjecturally agree with the capacities $c_k$ in the main four-dimensional cases.

\paragraph{Acknowledgments.} Thanks to Grisha Mikhalkin and Kyler Siegel for helpful discussion and comments on an initial draft of this note. Thanks to Brayan Ferreira for pointing out the example in Remark~\ref{rem:brayan}. Thanks to the anonymous referees for their careful reading and corrections.


\section{Definition of the capacities $c_k$}
\label{sec:definition}

We begin by recalling some basic definitions.

Let $Y$ be a three-manifold and let $\lambda$ be a contact form on $Y$. Let $\xi=\Ker(\lambda)$ denote the associated contact structure, and let $R$ denote the associated Reeb vector field. Define an {\em orbit set\/} to be a finite set of pairs $\alpha=\{(\alpha_i,m_i)\}$ where the $\alpha_i$ are distinct simple Reeb orbits, and the $m_i$ are positive integers. Define the {\em symplectic action\/} of the orbit set $\alpha$ by
\[
\mc{A}(\alpha) = \sum_im_i\int_{\alpha_i}\lambda.
\]
The contact form $\lambda$ is {\em nondegenerate\/} if every Reeb orbit (simple or multiply covered) is nondegenerate, i.e.\ the linearized return map does not have $1$ as an eigenvalue.

We say that an almost complex structure $J$ on $\R\times Y$ is {\em $\lambda$-compatible\/} if $J\partial_s=R$, where $s$ denotes the $\R$ coordinate; $J$ sends the contact structure $\xi$ to itself, rotating positively in the sense that $d\lambda(v,Jv)>0$ for every nonzero $v\in\xi$; and $J$ is $\R$-invariant.

We define a four-dimensional {\em Liouville domain\/} to be a compact symplectic four-manifold $(X,\omega)$ with boundary $Y$ such that there exists a primitive of $\omega$ which restricts to a contact form $\lambda$ on $Y$, for which the contact orientation of $Y$ agrees with the boundary orientation of $\partial X$. A basic example is a star-shaped domain in $\R^4$. Here a ``star-shaped domain'' is a compact domain in $\R^{2n}$ with smooth boundary which is transverse to the radial vector field, with the restriction of the standard symplectic form. We say that the Liouville domain $(X,\omega)$ is {\em nondegenerate\/} if the contact form $\lambda$ on $Y$ is nondegenerate; this notion does not depend on the choice of primitive of $\omega$.

Given a Liouville domain as above, and given $\epsilon>0$, a choice of primitive of $\omega$ determines a neighborhood $N_\epsilon$ of $Y$ in $X$, and an identification
\begin{equation}
\label{eqn:nepsilon}
N_\epsilon \simeq (-\epsilon,0]\times Y,
\end{equation}
under which $\omega|_{N_\epsilon}$ is identified with $d(e^s\lambda)$, where $s$ denotes the $(-\epsilon,0]$ coordinate. Using this identification, we can glue to obtain a smooth manifold
\begin{equation}
\label{eqn:completion}
\overline{X} = X\cup_Y ([0,\infty)\times Y),
\end{equation}
which we call the ``symplectization completion'' of $X$. This has a symplectic form $\overline{\omega}$ which agrees with $\omega$ on $X$ and with $d(e^s\lambda)$ on $[0,\infty)\times Y$. Strictly speaking, this completion depends on the choice of primitive of $\omega$, which we suppress from the notation.

We say that an almost complex structure $J$ on $\overline{X}$ is {\em cobordism-compatible\/} if $J|_X$ is $\omega$-compatible, and if $J|_{[0,\infty)\times Y}$ is the restriction of a $\lambda$-compatible almost complex structure on $\R\times Y$.

Define an {\em admissible symplectic four-manifold\/} to be a (possibly disconnected) compact symplectic four-manifold $(X,\omega)$ such that each component is either closed or a nondegenerate Liouville domain. Define $\overline{X}$ to be the union of the closed components and the symplectization completions of the Liouvile domain components. Define $\mc{J}(\overline{X},\omega)$ to be the set of almost complex structures on $\overline{X}$ which are $\omega$-compatible on the closed components and cobordism-compatible on the completed Liouville domain components.

Let $J\in\mc{J}(\overline{X},\omega)$. We consider holomorphic maps
\[
u: (\Sigma,j) \longrightarrow (\overline{X},J)
\]
where $\Sigma$ is a punctured compact Riemann surface (possibly disconnected), such that for each puncture in $\Sigma$, there is a Reeb orbit $\gamma$ on $\partial X$ and a neighborhood of the puncture mapping asymptotically to $[0,\infty)\times\gamma$ as $s\to\infty$. To avoid trivialities we assume that the restriction of $u$ to each component of the domain $\Sigma$ is nonconstant. Let $\M^J(\overline{X})$ denote the set of $J$-holomorphic maps as above, modulo reparametrization by biholomorphic maps $(\Sigma',j')\stackrel{\simeq}{\to}(\Sigma,j)$. If $x_1,\ldots,x_k\in X$ are distinct points, let $\M^J(\overline{X};x_1,\ldots,x_k)$ denote the set of $u\in\M^J(\overline{X})$ such that $x_1,\ldots,x_k\in u(\Sigma)$.

Define the {\em energy\/} $\mc{E}(u)$ as follows. If $\Sigma$ is connected and $u$ maps to a closed component of $X$, then $\mc{E}(u)=\int_\Sigma u^*\omega$. If $\Sigma$ is connected and $u$ maps to a completed Liouville domain component, then $\mc{E}(u)$ is the sum over the punctures of $\Sigma$ of the symplectic actions of the corresponding Reeb orbits. If $\Sigma$ is disconnected, then $\mc{E}(u)$ is the sum of the energies of the connected components.

\begin{definition}
\label{def:main}
Let $(X,\omega)$ be an admissible symplectic four-manifold and let $k$ be a nonnegative integer. Define
\begin{equation}
\label{eqn:maindef}
c_k(X,\omega) = 
\sup_{\substack{J\in\mc{J}(\overline{X})\\ \mbox{\scriptsize $x_1,\ldots,x_k\in X$ distinct}}} \inf_{u\in\mc{M}^J(\overline{X};x_1,\ldots,x_k)} \mc{E}(u) \in [0,\infty].
\end{equation}
\end{definition}

\begin{remark}
\label{rem:mc}
A key observation, which avoids various technical difficulties, is that in \eqref{eqn:maindef}, we can restrict attention to holomorphic curves $u$ that {\em do not have any multiply covered components\/}\footnote{We say that $u:\Sigma\to\overline{X}$ ``has no multiply covered components'' if the restriction of the map $u$ to each component of the domain $\Sigma$ is not multiply covered (which means that it must be somewhere injective), and no two components of $\Sigma$ have the same image under $u$.}.  This is because we can always replace a multiply covered component by the underlying somewhere injective curve to reduce energy without invalidating the point constraints.
\end{remark}

\begin{lemma}
\label{lem:monotonicity}
(proved in \S\ref{sec:monotonicity})
Let $(X,\omega)$ and $(X',\omega')$ be admissible symplectic four-manifolds and let $k$ be a nonnegative integer. If there exists a symplectic embedding $\varphi:(X,\omega)\to (X',\omega')$, then
\[
c_k(X,\omega) \le c_k(X',\omega').
\]
\end{lemma}

To extend the definition of $c_k$ to more general symplectic four-manifolds, we use the following basic procedure; compare \cite[\S4.2]{qech}.

\begin{definition}
\label{def:sup}
Let $(X',\omega')$ be any symplectic four-manifold  (possibly noncompact, disconnected, and/or with boundary or corners) and let $k$ be a nonnegative integer. Define
\[
c_k(X',\omega') = \sup\{c_k(X,\omega)\}
\]
where the supremum is over admissible symplectic four-manifolds $(X,\omega)$ for which there exists a symplectic embedding $\varphi:(X,\omega)\to (X',\omega')$.
\end{definition}

It follows from Lemma~\ref{lem:monotonicity} that Definition~\ref{def:sup} agrees with Definition~\ref{def:main} when $(X',\omega')$ is already an admissible symplectic four-manifold.

\begin{remark}
\label{rem:ms}
The definition of $c_k$ is inspired by the paper of McDuff-Siegel \cite{ms}, which gives a similar elementary definition of a sequence of symplectic capacities $\widetilde{\frak{g}}_k$, as an alternative to symplectic capacities that were defined in \cite{siegel} using rational SFT \cite{egh}. The capacities $\widetilde{\frak{g}}_k$ are defined for symplectic manifolds of any dimension using genus zero holomorphic curves that are constrained to have contact of order $k$ with a local divisor.

Some variants of Definition~\ref{def:main} are possible. For example one could require each component of the domain of $u$ to have genus zero; the resulting capacities may be related to the capacities $\widetilde{\frak{g}}_k$.
\end{remark}


\section{Proof of the monotonicity lemma}
\label{sec:monotonicity}

Before discussing the basic properties of the capacities $c_k$ in \S\ref{sec:properties} below, we now prove Lemma~\ref{lem:monotonicity}.

The following notation will be useful. Let $(X,\omega)$ be an admissible symplectic four-manifold and let $\alpha=\{(\alpha_i,m_i)\}$ be an orbit set for $Y=\partial X$. Let $H_2(X,\alpha)$ denote the set of relative homology classes of $2$-chains $Z$ in $X$ with $\partial Z = \alpha$. This set is an affine space over $H_2(X)$.

Given $J\in\mc{J}(\overline{X},\omega)$, let $\M^J(\overline{X},\alpha;x_1,\ldots,x_k)$ denote the set of holomorphic curves in $\M^J(\overline{X};x_1,\ldots,x_k)$ such that for each $i$, there are punctures asymptotic to covers of $\alpha_i$ with total multiplicity $m_i$, and there are no other punctures.
Note that each $u\in\M^J(\overline{X},\alpha;x_1,\ldots,x_k)$ has a well-defined relative homology class $[u]\in H_2(X,\alpha)$.

\begin{proof}[Proof of Lemma~\ref{lem:monotonicity}.]
For $\epsilon>0$, let $N_\epsilon$ denote the neighborhood of $\partial X$ in \eqref{eqn:nepsilon}. The time $\epsilon$ flow of the Liouville vector field (coming from the primitive of $\omega$) defines a symplectomorphism
\begin{equation}
\label{eqn:shrink}
(X\setminus N_\epsilon, \omega|_{X\setminus N_\epsilon}) \simeq (X,e^{-\epsilon}\omega).
\end{equation}
It follows from Definition~\ref{def:main} that $c_k$ satisfies the Conformality property, so we deduce from \eqref{eqn:shrink} that
\begin{equation}
\label{eqn:shrinkingcapacity}
c_k(X\setminus N_\epsilon, \omega|_{X\setminus N_\epsilon}) = e^{-\epsilon}c_k(X,\omega)
\end{equation}
Consequently, by replacing $X$ with $X\setminus N_\epsilon$ for $\epsilon>0$ small if necessary, we can assume without loss of generality that $\varphi(X)\subset \op{int}(X')$.

Now fix $x_1,\ldots,x_k\in X$ distinct, $J\in\mc{J}(\overline{X},\omega)$, and $\epsilon>0$. To prove the lemma, we need to show that there exists $u\in\M^J(\overline{X};x_1,\ldots,x_k)$ with
\[
\mc{E}(u) < c_k(X',\omega') + \epsilon.
\]
We will use a ``neck stretching'' argument.

Write $Y=\partial X$ and let $\lambda$ denote the contact form on $Y$. Since $\varphi(X)\subset\op{int}(X')$, there exists a neighborhood $\mc{U}$ of $\varphi(Y)$ in $X'\setminus\varphi(\op{int}(X))$ and an identification
\[
(\mc{U},\omega'|_\mc{U}) \simeq ([0,\delta)\times Y,d(e^s\lambda))
\]
for some $\delta>0$, where $s$ denotes the $[0,\delta)$ coordinate. For each $R>0$, we can choose an almost complex structure $J_R\in\mc{J}(\overline{X'},\omega')$ such that $\varphi$ extends to a biholomorphism
\begin{equation}
\label{eqn:bih}
\varphi_R: (X\cup_Y([0,R)\times Y),J) \stackrel{\simeq}{\longrightarrow} (\varphi(X)\cup\mc{U},J_R).
\end{equation}
We can further assume that $J_R$ is independent of $R$ outside of $\varphi(X)\cup\mc{U}$.

By the definition of $c_k$, for each $R$ we can choose
\[
u_R\in\M^{J_R}(\overline{X'};\varphi(x_1),\ldots,\varphi(x_k))
\]
with
\begin{equation}
\label{eqn:energybound}
\mc{E}(u_R) < c_k(X',\omega') + \epsilon.
\end{equation}
Let $u_R^\varphi$ denote the intersection of the curve $u_R$ with $\varphi(X)\cup\mc{U}$, composed with $\varphi_R^{-1}$.
We now want to argue that there is a sequence $R_i\to\infty$ such that the intersections of the curves $u_{R_i}^\varphi$ converge in some sense to the desired curve $u$. This task is complicated by the fact that we do not have an a priori bound on the genus of the components of the domains of the curves $u_R$, so we cannot directly use SFT compactness as in \cite{behwz,cm}.

Fortunately, there is a local version of Gromov compactness using currents which does not require any genus bound. This was proved in the four-dimensional case by Taubes \cite[Prop.\ 3.3]{taubesgc}, and an updated version which works in arbitrary dimension was proved by Doan-Wapulski \cite[Prop.\ 1.9]{dw}. By this local Gromov compactness and the energy bound \eqref{eqn:energybound}, as applied in \cite[\S9.4]{pfh2}, we can find a sequence $R_i\to\infty$ such that the curves $u_{R_i}^\varphi$ converge as currents to a proper holomorphic map $u$ to $\overline{X}$ which passes through the points $x_1,\ldots,x_k$, is asymptotic as a current as $s\to\infty$ to an orbit set for $Y$, and has energy less than $c_k(X',\omega') + \epsilon$. A priori, components of the domain of $u$ may have infinite genus, and to complete the proof of the lemma we need to arrange that they are punctured compact Riemann surfaces.

We can choose the sequence $R_i$ so that there is a single orbit set $\alpha'$ for $\partial X'$ such that each $u_{R_i}$ is in $\mc{M}^{J_{R_i}}(\overline{X'},\alpha';\varphi(x_1),\ldots,\varphi(x_k))$, because there are only finitely many orbit sets for $\partial X'$ with action less than $c_k(X',\omega') + \epsilon$. When applying Gromov compactness above, we can further use the arguments in \cite[\S9.4]{pfh2} to chase down the rest of the energy of the holomorphic curves $u_{R_i}$ and pass to a subsequence such that the relative homology class $[u_{R_i}]\in H_2(X',\alpha')$ does not depend on $i$.

By Remark~\ref{rem:mc}, we can assume that each $u_{R_i}$ has no multiply covered components. Since we are in four dimensions, the relative adjunction formula of \cite[Prop.\ 4.9]{ir} and the asymptotic writhe bound of \cite[Lem.\ 4.20]{ir} imply that there is a lower bound on the Euler characteristic of the domain of $u_{R_i}$ depending only on the orbit set $\alpha'$ and the relative homology class $[u_{R_i}]$. We can also assume that the domain of each $u_{R_i}$ has at most $k$ components, since otherwise some components can be discarded without violating the requirement to pass through the points $x_1,\ldots,x_k$. Consequently we obtain an $i$-independent upper bound on the genus of each component of the domain of $u_{R_i}$.

We can then pass to a subsequence such that the components of the domain of $u_{R_i}$ can be numbered so that the $j^{th}$ component is a punctured compact Riemann surface with the genus and number of punctures not depending on $i$, and the sequence of $j^{th}$ components with the restrictions of the maps $u_{R_i}$ converges as $i\to\infty$ to a component of $u$ whose domain is also a punctured compact Riemann surface.
\end{proof}


\section{Properties of the capacities $c_k$}
\label{sec:properties}

\begin{theorem}
\label{thm:properties}
The capacities $c_k$ of four-dimensional symplectic manifolds have the following properties:
\begin{description}
\item{(Conformality)}
If $r>0$ then
\begin{equation}
\label{eqn:conformality}
c_k(X,r\omega) = r c_k(X,\omega).
\end{equation}
\item{(Increasing)}
\[
0=c_0(X,\omega) < c_1(X,\omega) \le c_2(X,\omega) \le \cdots \le +\infty.
\]
\item{(Disjoint Union)}
\[
c_k\left(\coprod_{i=1}^m(X_i,\omega_i)\right) = \max_{k_1+\cdots+k_m=k}\sum_{i=1}^m c_{k_i}(X_i,\omega_i).
\]
\item{(Sublinearity)}
\[
c_{k+l}(X,\omega) \le c_k(X,\omega) + c_l(X,\omega).
\]
\item{(Monotonicity)}
If there exists a symplectic embedding $\varphi: (X,\omega)\to (X',\omega')$, then
\[
c_k(X,\omega) \le c_k(X',\omega').
\]
\item{($C^0$-Continuity)}
For each $k$, the capacity $c_k$ defines a continuous function on the set of star-shaped domains in $\R^{4}$ with respect to the Hausdorff metric on compact sets.
\item{(Spectrality)}
If $(X,\omega)$ is a four-dimensional Liouville domain with boundary $Y$, then for each $k$ with $c_k(X,\omega)<\infty$, there exists an orbit set $\alpha$ in $Y$, which is nullhomologous in $X$, with $c_k(X,\omega)=\mc{A}(\alpha)$.
\item{(ECH Index)}
If $X$ is a nondegenerate star-shaped domain in $\R^4$, then $c_k(X)<\infty$, and in the Spectrality property, we can choose $\alpha$ so that its ECH index\footnote{See e.g.\ \cite[Def.\ 5.2]{ruelle} for the definition of the ECH index of $\alpha$. The definition there is stated for ECH generators (a special kind of orbit set, see Remark~\ref{rem:generator}), but is valid for arbitrary orbit sets.} satisfies $I(\alpha)\ge 2k$.
\item{(Ball)}
\[
c_k(B^4(a))=da
\]
where $d$ is the unique nonnegative integer with
\[
d^2+d \le 2k \le d^2+3d.
\]
\item{(Asymptotics)}
If $X\subset\R^4$ is a compact domain with smooth boundary, then
\[
c_k(X) = 2\op{vol}(X)^{1/2}k^{1/2} + O(k^{1/4}).
\]
\end{description}
\end{theorem}

\begin{proof}
For admissible symplectic four-manifolds, the Conformality, Increasing, Disjoint Union, and Sublinearity properties follow immediately from Definition~\ref{def:main}. It then follows from Lemma~\ref{lem:monotonicity} and Definition~\ref{def:sup} that these properties, as well as the Monotonicity property, also hold for general symplectic four-manifolds. 

The $C^0$-Continuity property follows from Conformality and Monotonicity, since if two star-shaped domains are close in the Hausdorff metric, then each is contained in the scaling of the other by a number slightly larger than $1$. Note here that if $X$ is a star-shaped domain and $r>0$, then Conformality implies that $c_k(rX)=r^2c_k(X)$.

To prove the Spectrality property, suppose first that $(X,\omega)$ is a nondegenerate Liouville domain with $c_k(X,\omega)<\infty$. It follows from the definition of $c_k$ that there is an orbit set $\alpha$ with $c_k(X,\omega)=\mc{A}(\alpha)$, because in \eqref{eqn:maindef}, for every curve $u$, the energy $\mc{E}(u)$ is the action of some orbit set $\alpha$, and the set of all such actions is discrete. Also $\alpha$ is nullhomologous in $X$ because there is a holomorphic curve in $\overline{X}$ asymptotic to it.

If $(X,\omega)$ is a degenerate Liouville domain, then the Spectrality property follows by approximating with nondegenerate Liouville domains and using \eqref{eqn:shrinkingcapacity} and Monotonicity as in the proof of $C^0$ continuity.

To prove the ECH Index property, first note that $c_k(X)<\infty$ by Monotonicity and the upper bound on $c_k$ of a ball proved in \eqref{eqn:ballub} below. Recall from Remark~\ref{rem:mc} that in \eqref{eqn:maindef}, we can restrict attention to holomorphic curves that do not have any multiply covered components. Let $\M^J_*(\overline{X},\alpha;x_1,\ldots,x_k)$ denote the set of curves in $\M^J(\overline{X},\alpha;x_1,\ldots,x_k)$ without multiply covered components. The hypothesis that $X$ is nondegenerate implies that the set of symplectic actions of orbit sets in $\partial X$ is discrete, so we can rewrite \eqref{eqn:maindef} as
\begin{equation}
\label{eqn:rewrite}
c_k(X) = \max_{\substack{J\in\mc{J}(\overline{X})\\ \mbox{\scriptsize $x_1,\ldots,x_k\in X$ distinct}}}\min \left\{\mc{A}(\alpha) \;\big|\;  \M^J_*\left(\overline{X},\alpha;x_1,\ldots,x_k\right) \neq \emptyset\right\}.
\end{equation}

If $u\in\M^J_*(\overline{X},\alpha;x_1,\ldots,x_k)$, then it follows from the ECH index inequality, see e.g.\ \cite[\S3.4]{bn}, that
\begin{equation}
    \label{eqn:indexinequality}
\op{ind}(u) \le I(\alpha).
\end{equation}
Here $\op{ind}(u)$ denotes the Fredholm index of $u$, which for generic $J\in\mc{J}(\overline{X})$ is the dimension of the component of the moduli space $\M^J_*(\overline{X},\alpha)$ containing $u$. In particular, if $J\in\mc{J}(\overline{X})$ and $x_1,\ldots,x_k\in X$ are generic, then for any $u\in \M^J_*(\overline{X},\alpha;x_1,\ldots,x_k)$, the dimension of the component of the latter moduli space containing $u$ is $\op{ind}(u) - 2k \ge 0$, so if the latter moduli space is nonempty then by \eqref{eqn:indexinequality} we have $I(\alpha)\ge 2k$.
It follows that for generic $J\in\mc{J}(\overline{X})$ and $x_1,\ldots,x_k\in X$, the minimum in \eqref{eqn:rewrite} has the form $\mc{A}(\alpha)$ where $I(\alpha)\ge 2k$. By Gromov compactness as in the proof of Lemma~\ref{lem:monotonicity}, the maximum in \eqref{eqn:rewrite} must be realized by generic $J\in\mc{J}(\overline{X})$ and $x_1,\ldots,x_k\in X$.

To prepare for the proof of the Ball property, if $a,b>0$, define the ellipsoid
\[
E(a,b)=\left\{z\in\C^2 \;\bigg|\; \frac{\pi|z_1|^2}{a} + \frac{\pi|z_2|^2}{b} \le 1\right\}.
\]
Calculations e.g.\ in \cite[\S3.7]{bn} show that for any ellipsoid $E(a,b)$ with $a/b$ irrational, there are just two simple Reeb orbits, which have symplectic action $a$ and $b$, and the ECH index defines a bijection from the set of orbit sets to the set of nonnegative even integers. Furthermore the symplectic action is an increasing function of the ECH index. 

To prove the Ball property, by the Conformality property we can assume that $a=1$. Let $\epsilon>0$ be irrational and consider the ellipsoid
\[
E(1-\epsilon,1) \subset E(1,1) = B^4(1).
\]
For a given nonnegative integer $d$, if $\epsilon$ is sufficiently small, then by the previous paragraph, the orbit set of ECH index $d^2+d$ has symplectic action $d(1-\epsilon)$. Taking $\epsilon\to 0$, it follows from the ECH index and Monotonicity properties that 
\begin{equation}
\label{eqn:balllb}
c_{(d^2+d)/2}(B^4(1)) \ge d.
\end{equation}

To complete the proof of the Ball property, by the Increasing property, we need to show that
\begin{equation}
\label{eqn:ballub}
c_{(d^2+3d)/2}(B^4(1)) \le d.
\end{equation}
By Monotonocity, it is enough to show that
\begin{equation}
\label{eqn:fs}
c_{(d^2+3d)/2}(\C P^2,\omega_{FS})\le d.
\end{equation}
Here $\omega_{FS}$ denotes the Fubini-Study form on $\C P^2$, normalized so that a line has symplectic area $1$. To prove \eqref{eqn:fs}, write $k=(d^2+3d)/2$; it is enough to show that for any $J\in\mc{J}(\C P^2,\omega_{FS})$ and any $x_1,\ldots,x_k\in\C P^2$, there exists a $J$-holomorphic curve, possibly with disconnected domain, of total degree $d$ passing through the points $x_1,\ldots,x_k$. For a given $J$, for generic $x_1,\ldots,x_k$ this was shown by Gromov \cite[\S0.2.B]{nonsqueezing} (it also follows from Taubes's ``Seiberg-Witten = Gromov'' theorem as explained in the proof of Theorem~\ref{thm:closed} below), and for arbitrary $x_1,\ldots,x_k$ it follows from Gromov compactness.

Finally, the Asymptotics property was shown for ECH capacities in \cite[Thm.\ 1.1]{ruelle}. The proof there just uses the Monotonicity and Disjoint Union properties for ECH capacities and the formula for the ECH capacities of a cube. Theorem~\ref{thm:convex} below implies that for a cube, the ECH capacities and the capacities $c_k$ agree. Hence the Asymptotics property also holds for the capacities $c_k$.
\end{proof}

\begin{remark}
\label{rem:generator}
The properties of the capacities $c_k$ in Theorem~\ref{thm:properties}, aside from the Sublinearity property, are also known to hold for ECH capacities. These properties of ECH capacities were proved in \cite{qech}, except for the Asymptotics property, which is a later refinement proved in \cite{ruelle}.

For the ECH capacities, a slighty stronger version of the ECH Index property follows from the definition of ECH capacities reviewed in \S\ref{sec:compare} below: namely one can arrange that $I(\alpha)=2k$, and furthermore that the orbit set $\alpha$ is an ECH generator. Here we say that an orbit set $\alpha=\{(\alpha_i,m_i)\}$ is an {\em ECH generator\/} if $m_i=1$ whenever $\alpha_i$ is hyperbolic (meaning that the linearized return map has real eigenvalues).
\end{remark}

\begin{remark}
Some applications of ECH capacities only need the properties in Theorem~\ref{thm:properties}, and thus can be re-proved using the capacities $c_k$. For example, Irie \cite{irie} proved a $C^\infty$ closing lemma for Reeb vector fields on closed three-manifolds, using the asymptotics of the ECH spectrum \cite{vc}. In the case of $S^3$ with the standard contact structure, which corresponds to star-shaped hypersurfaces in $\R^4$, the ECH spectrum agrees with the ECH capacities of the corresponding star-shaped domain, and Irie's proof of the closing lemma works using only the $C^0$-Continuity, Spectrality, and Asymptotics properties in Theorem~\ref{thm:properties}.
\end{remark}


\section{Computation for convex toric domains}
\label{sec:toric}

We now show that for ``convex toric domains'', the capacities $c_k$ agree with a known combinatorial formula for their ECH capacities\footnote{This formula appears in \cite[Prop.\ 5.6]{beyond}. It is a specialization of a result in \cite[Cor.\ A.12]{concaveconvex} computing the ECH capacities of a more general notion of ``convex toric domain''.}. In fact, the capacities $c_k$ for these examples are uniquely determined by the properties in Theorem~\ref{thm:properties}.

Let $\Omega$ be a compact domain in $\R^2_{\ge 0}$. Define the {\em toric domain\/}
\[
X_\Omega = \left\{z\in\C^n \;\big|\; \pi\left(|z_1|^2,|z_2|^2\right)\in\Omega\right\}.
\]
Define a (four-dimensional) {\em convex toric domain\/} to be a toric domain $X_\Omega$ as above such that the set
\[
\widehat{\Omega} = \left\{\mu\in\R^2 \;\big|\; \left(|\mu_1|,|\mu_2|\right)\in\Omega\right\}
\]
is convex\footnote{This is slightly misleading terminology, as a ``convex toric domain'' is not the same thing as a toric domain that is convex; see \cite[\S2]{ghr} for clarification.}. Define a (four-dimensional) {\em concave toric domain\/} to be a toric domain $X_\Omega$ such that the set $\R^2_{\ge 0}\setminus\Omega$ is convex.

If $X_\Omega$ is a four-dimensional convex toric domain, let $\|\cdot\|_\Omega^*$ denote the norm on $\R^2$ defined by
\[
\|v\|_\Omega^* = \max\left\{\langle v,w\rangle \;\big|\; w\in\widehat{\Omega}\right\}.
\]
If $\gamma:[\alpha,\beta]\to\R^2$ is a continuous, piecewise differentiable curve, define its {\em $\Omega$-length\/} by
\begin{equation}
\label{eqn:omegalength}
\ell_\Omega(\gamma) = \int_\alpha^\beta \|J\gamma'(t)\|_\Omega^*\,dt
\end{equation}
where $J=\begin{pmatrix} 0 & -1 \\ 1 & 0 \end{pmatrix}$.

Define a {\em convex integral path\/} to be a polygonal path $\Lambda$ in the nonnegative quadrant from the point $(0,b)$ to the point $(a,0)$, for some nonnegative integers $a$ and $b$, with vertices at lattice points, such that the region bounded by $\Lambda$ and the line segments from $(0,0)$ to $(a,0)$ and from $(0,0)$ to $(0,b)$ is convex. Define $\widehat{\mc{L}}(\Lambda)$ to be the number of lattice points in this region, including lattice points on the boundary.

\begin{theorem}
\label{thm:convex}
If $X_\Omega$ is a four-dimensional convex toric domain, then
\begin{equation}
\label{eqn:convex}
c_k(X_\Omega) = \min\{\ell_\Omega(\Lambda) \mid \widehat{\mc{L}}(\Lambda)= k+1\}
\end{equation}
where the minimum is over convex integral paths $\Lambda$.
\end{theorem}

\begin{proof}
It is shown in \cite[Lem.\ 5.4]{beyond} that given $L,\epsilon>0$, there is a nondegenerate star-shaped domain $X'$ with $\op{dist}_{C^0}(X',X_\Omega)<\epsilon$ with the following property: Every orbit set $\alpha$ for $X'$ with action $\mc{A}(\alpha)<L$ determines a convex integral path $\Lambda$ such that $|\mc{A}(\alpha)-\ell_\Omega(\Lambda)|<\epsilon$ and the ECH index $I(\alpha)\le 2(\widehat{\mc{L}}(\Lambda)-1)$. It then follows from the $C^0$-Continuity and ECH Index properties in Theorem~\ref{thm:properties} that
\begin{equation}
\label{eqn:ineq1}
c_k(X_\Omega) \ge \min\{\ell_\Omega(\Lambda) \mid \widehat{\mc{L}}(\Lambda)\ge k+1\}.
\end{equation}

We now prove the reverse inequality. Let $a>0$ be the smallest real number such that $X_\Omega\subset B^4(a)$. In \cite[\S2.2]{concaveconvex}, a ``negative weight sequence'' is defined; this is a nonincreasing (possibly finite) sequence of positive real numbers $(a_1,a_2,\ldots)$. It has the property that there is a symplectic embedding
\[
X_\Omega \sqcup \coprod_i\op{int}(B^4(a_i)) \longrightarrow B^4(a)
\]
which fills the volume of $B^4(a)$. It follows from the Disjoint Union property that
\[
c_k(X_\Omega) \le \inf_{l\ge 0}\left(c_{k+l}(B^4(a)) - c_l\left(\coprod_{i\le l}B^4(a_i)\right)\right).
\]
Furthermore, $c_k$ agrees with $c_k^{\op{ECH}}$ for a disjoint union of balls by the Disjoint Union and Ball properties, so we can rewrite the above inequality as
\begin{equation}
\label{eqn:ineq2}
c_k(X_\Omega) \le \inf_{l\ge 0}\left(c_{k+l}^{\op{ECH}}(B^4(a)) - c_l^{\op{ECH}}\left(\coprod_{i\le l}B^4(a_i)\right)\right).
\end{equation}
Finally, a combinatorial calculation in \cite[\S A.3]{concaveconvex} shows that the right hand side of \eqref{eqn:ineq2} is less than or equal to the right hand side of \eqref{eqn:ineq1}.

To complete the proof, we observe that
\[
\min\{\ell_\Omega(\Lambda) \mid \widehat{\mc{L}}(\Lambda)\ge k+1\} = \min\{\ell_\Omega(\Lambda) \mid \widehat{\mc{L}}(\Lambda) = k+1\},
\]
as explained in \cite[\S A.3]{concaveconvex}.
\end{proof}

\begin{remark}
\label{rem:beyond}
By Theorem~\ref{thm:convex} and \cite[Prop.\ 5.6]{beyond}, the capacities $c_k$ agree with the ECH capacities for convex toric domains. It follows from the Monotonicity property that all obstructions to symplectic embeddings between convex toric domains coming from ECH capacities can be recovered using the capacities $c_k$.
\end{remark}

\begin{remark}
Going beyond ECH capacities, it is shown in \cite[Thm.\ 1.19]{beyond} that if $X_\Omega$ and $X_{\Omega'}$ are four-dimensional convex toric domains, and if there exists a symplectic embedding $X_\Omega\to X_{\Omega'}$, then a certain combinatorial criterion holds. This leads to stronger symplectic embedding obstructions in some cases where ECH capacities do not give sharp obstructions, for example to symplectically embedding a polydisk into a ball or ellipsoid; see \cite{beyond,seven,pebble}.

The proof of \cite[Thm.\ 1.19]{beyond} rests on the existence of an ECH index 0 holomorphic curve with certain properties in (the completion of) a symplectic cobordism between the (perturbed) boundaries of $X_\Omega$ and $X_{\Omega'}$, which is produced using Seiberg-Witten theory. One can re-prove the existence of such a curve using the methods of this paper, namely by using the existence of curves in $\overline{X_{\Omega'}}$ with point constraints in the image of $X_{\Omega}$, as guaranteed by the capacities $c_k$, and then neck stretching as in the proof of Lemma~\ref{lem:monotonicity}.
\end{remark}


\section{Comparison with ECH capacities}
\label{sec:compare}

Aside from the examples of toric domains, we do not know to what extent $c_k$ agrees with $c_k^{\op{ECH}}$, but we do have the following general fact, whose proof (and statement) use Seiberg-Witten theory:

\begin{theorem}
\label{thm:compare}
Let $X$ be a four-dimensional Liouville domain and let $k$ be a nonnegative integer. Then
\[
c_k(X) \le c_k^{\op{ECH}}(X).
\]
\end{theorem}

To prepare for the proof of Theorem~\ref{thm:compare}, we now recall the definition of the ECH capacities $c_k^{\op{ECH}}$, for the simplest case of four-dimensional nondegenerate Liouville domains with connected boundary.

Let $Y$ be a closed oriented three-manifold and let $\lambda$ be a nondegenerate contact form on $Y$.  The following is an outline of the definition of the {\em embedded contact homology\/} $ECH(Y,\lambda)$. We define $ECC(Y,\lambda)$ to be the free $\Z/2$-module\footnote{It is also possible to define ECH with integer coefficients \cite[\S9]{obg2}.} generated by the ECH generators; see Remark~\ref{rem:generator}. For a generic $\lambda$-compatible almost complex structure $J$ on $\R\times Y$, the ECH differential
\[
\partial_J:ECC(Y,\lambda)\longrightarrow ECC(Y,\lambda)
\]
is defined as follows. If $\alpha$ and $\beta$ are ECH generators, then the coefficient of $\beta$ in $\partial_J\alpha$, which we denote by $\langle\partial_J\alpha,\beta\rangle \in \Z/2$, is a mod 2 count of ``$J$-holomorphic currents'' $\mc{C}$ in $\R\times Y$, modulo $\R$ translation, that are asymptotic to $\alpha$ as $s\to+\infty$ and to $\beta$ as $s\to-\infty$, and that have ECH index $I(\mc{C})=1$. See \cite[\S3]{bn} for detailed definitions. It is shown in \cite{obg1} that $\partial_J^2=0$. We define $ECH(Y,\lambda)$ to be the homology of the chain complex $(ECC(Y,\lambda),\partial_J)$.

It follows from the definition of $\lambda$-compatible almost complex structure that the ECH differential decreases symplectic action:
\begin{equation}
\label{eqn:dda}
\langle\partial_J\alpha,\beta\rangle\neq 0 \Longrightarrow \mc{A}(\alpha)>\mc{A}(\beta).
\end{equation}
As a result, for each $L\in\R$, the ECH generators with action less than $L$ span a subcomplex of $(ECC(Y,\lambda),\partial_J)$. We define the {\em filtered ECH}, which we denote by $ECH^L(Y,\lambda)$, to be the homology of this subcomplex.

It was shown by Taubes \cite{taubes} that $ECH(Y,\lambda)$ is isomorphic to a version of Seiberg-Witten Floer cohomology defined by Kronheimer-Mrowka \cite{km}. Taubes's isomorphism was used in \cite[Thm.\ 1.3]{cc2} to show that $ECH(Y,\lambda)$ and $ECH^L(Y,\lambda)$ do not depend on $J$; that is, the homologies for different choices of $J$ are canonically isomorphic to each other.

There is also a map
\[
U:ECH^L(Y,\lambda)\longrightarrow ECH^L(Y,\lambda)
\]
induced by a chain map which counts $J$-holomorphic currents with ECH index 2 passing through a base point in $\R\times Y$. This map does not depend on the choice of base point when $Y$ is connected; otherwise it depends on a choice of connected component of $Y$. See \cite[\S2.5]{wh} for more details.

Now let $(X,\omega)$ be a four-dimensional nondegenerate Liouville domain with connected boundary $Y$ and associated contact form $\lambda$. In this case the $k^{th}$ ECH capacity is defined by
\begin{equation}
\label{eqn:defckech}
c_k^{\op{ECH}}(X,\omega) = \inf\left\{L\ge 0 \;\big|\; \exists \eta\in ECH^L(Y,\lambda): U^k\eta=[\emptyset]\right\}.
\end{equation}
Here $[\emptyset]$ is the homology class in $ECH^L(Y,\lambda)$ of the empty set of Reeb orbits, which is a cycle by \eqref{eqn:dda}. Note that by \eqref{eqn:cobordismmap} below, the existence of an exact filling of $Y$ (namely the Liouville domain $X$) implies that the class $[\emptyset]\neq 0$ in $ECH^L(Y,\lambda)$.

\begin{proof}[Proof of Theorem~\ref{thm:compare}.]
Let $Y$ denote the boundary of $X$. For brevity we just explain the case when $Y$ is connected; the general case follows by a similar argument using the more general definition of ECH capacities in \cite[Def.\ 4.3]{qech}.

Since $c_k$ and $c_k^{\op{ECH}}$ both satisfy Conformality and Monotonicity, by a continuity argument using \eqref{eqn:shrinkingcapacity} and the analogous equation for $c_k^{\op{ECH}}$, we can assume without loss of generality that $X$ is nondegenerate.

Let $\lambda$ denote the contact form on $Y$. As explained for example in \cite[Thm.\ 2.3]{qech}, for each $L\ge 0$, the exact filling $X$ of $Y$ induces a cobordism map
\begin{equation}
    \label{eqn:cobordismmap}
\Phi:ECH^L(Y,\lambda) \longrightarrow \Z/2,
\end{equation}
defined using Seiberg-Witten theory, which sends $[\emptyset]$ to $1$. 

Now suppose that $J\in\mc{J}(\overline{X})$ and $x_1,\ldots,x_k\in X$. Heuristically one might expect that if $J$ and $x_1,\ldots,x_k$ are generic, then the composition
\begin{equation}
\label{eqn:composition}
\Phi\circ U^k:ECH^L(Y,\lambda)\longrightarrow \Z/2
\end{equation}
is induced by a cocycle
\[
\phi: ECC^L(Y,\lambda)\longrightarrow \Z/2
\]
that counts $J$-holomorphic curves in $\overline{X}$ with ECH index $2k$ passing through $x_1,\ldots,x_k$. What one can actually prove, as in the ``holomorphic curves axiom'' for ECH cobordism maps in \cite[Thm.\ 1.9]{cc2} and the comparison of $U$ maps in \cite[Thm.\ 1.1]{taubes5}, is the following. For any $J\in\mc{J}(\overline{X})$ and any $x_1,\ldots,x_k\in X$ (not necessarily generic), the composition in \eqref{eqn:composition} is induced by a (noncanonical) cocycle $\phi$ with the following property: If $\alpha$ is an ECH generator and $\phi(\alpha)\neq 0$, then there exists a ``broken $J$-holomorphic current'' in $\overline{X}$ passing through $x_1,\ldots,x_k$. This last statement implies that there is an orbit set $\alpha'$ with $\mc{A}(\alpha')\le \mc{A}(\alpha)$ and a holomorphic curve in $\M^J(\overline{X},\alpha';x_1,\ldots,x_k)$.

Now suppose that $L>c_k^{\op{ECH}}(X)$. Then by the definition of ECH capacities in \eqref{eqn:defckech}, there exists $\eta\in ECH^L(Y,\lambda)$ with $U^k\eta=[\emptyset]$. It follows that $(\Phi\circ U^k)(\eta)=1$. By the previous paragraph, for any $J\in\mc{J}(X)$ and any $x_1,\ldots,x_k\in X$, there exists an ECH generator $\alpha'$ with $\mc{A}(\alpha')<L$ such that $\M^J(X,\alpha';x_1,\ldots,x_k)\neq\emptyset$. It then follows from \eqref{eqn:rewrite} that $c_k(X)\le L$. Since $L>c_k^{\op{ECH}}(X)$ was arbitrary, the theorem follows.
\end{proof}

\begin{remark}
One can understand the inequality in Theorem~\ref{thm:compare} as follows: The number $c_k(X)$ measures the minimal energy of holomorphic curves in $\overline{X}$ through $k$ points that are guaranteed to exist, for whatever reason. On the other hand, $c_k^{\op{ECH}}(X)$ measures the energy of certain holomophic curves in $\overline{X}$ through $k$ points that are guaranteed to exist for ECH reasons.
\end{remark}

\begin{remark}
\label{rem:brayan}
There exist examples of Liouville domains and positive integers $k$ for which $c_k$ is strictly less than $c_k^{\op{ECH}}$. An example is given by the unit cotangent bundle $D^*S^2(4\pi)$, where $S^2(a)$ denotes the $2$-sphere with the round metric of area $a$. It follows from results in \cite{fr,ou} that there exist symplectic embeddings
\[
\op{int}(P(2\pi,2\pi)) \longrightarrow \op{int}(D^*S^2(4\pi)) \longrightarrow S^2(2\pi)\times S^2(2\pi).
\]
Here the left hand side is a polydisk; see equation \eqref{eqn:polydisk} for the notation. We will see in Remark~\ref{rem:closed} below that the capacities $c_k$ are the same for $P(2\pi,2\pi)$ and $S^2(2\pi)\times S^2(2\pi)$, so by Monotonicity they are also the same for $D^*S^2(4\pi)$. However the ECH capacities $c_k^{\op{ECH}}(D^*S^2(4\pi))$ are computed in \cite{fr} and found to be larger for some $k$.

The main reason for the discrepancy is the following: The Spectrality property in Theorem~\ref{thm:properties} asserts that $c_k$ of a Liouville domain $X$ with boundary $Y$ is the action of an orbit set which is nullhomologous in $X$. However by the definition of the ECH capacities in \eqref{eqn:defckech}, $c_k^{\op{ECH}}(X)$ is the action of an orbit set which is nullhomologous in $Y$, a more restrictive condition.
\end{remark}


\section{Additional computations using Seiberg-Witten theory}
\label{sec:SW}

We now compute some additional examples of the capacities $c_k$ using Seiberg-Witten theory (which could perhaps be avoided with more work). 

If $X_\Omega$ is a four-dimensional concave toric domain as defined above, define an ``anti-norm'' on $\R^2$ by
\[
[v]_\Omega = \min\{\langle (|v_1|,|v_2|),w\rangle \mid w\in\partial_+\Omega\}
\]
where $\partial_+\Omega$ denotes the closure of the portion of $\partial\Omega$ not on the axes. If $\gamma$ is a continuous, piecewise differentiable curve in $\R^2$, now define its $\Omega$-length as in \eqref{eqn:omegalength}, but replacing the norm $\|\cdot\|$ by the anti-norm $[\cdot]$.

Define a {\em concave integral path\/} to be a polygonal path $\Lambda$ in the nonnegative quadrant from the point $(0,b)$ to the point $(a,0)$, for some nonnegative integers $a$ and $b$, with vertices at lattice points, which is the graph of a convex function. Define $\check{\mc{L}}(\Lambda)$ to be the number of lattice points in the region bounded by $\Lambda$ and the axes, this time (in contrast to the case of convex toric domains) not including lattice points on $\Lambda$.

\begin{theorem}
\label{thm:concave}
If $X_\Omega$ is a four-dimensional concave toric domain, then
\begin{equation}
\label{eqn:concave}
c_k(X_\Omega) = \max\{\ell_\Omega(\Lambda) \mid \check{\mc{L}}(\Lambda)=k\}
\end{equation}
where the maximum is over concave integral paths $\Lambda$.
\end{theorem}

\begin{remark}
\label{rem:ECHconcave}
It is shown in  \cite[Thm.\ 1.21]{concave} that the same formula holds for the ECH capacities $c_k^{\op{ECH}}(X_\Omega)$.
\end{remark}

\begin{proof}[Proof of Theorem~\ref{thm:concave}.]
In \cite[\S1.3]{concave}, see also \cite[\S1.3]{ruelle}, a ``weight expansion'' of $X_\Omega$ is defined; this is a nonincreasing (possibly finite) sequence of positive real numbers $(a_1,a_2,\ldots)$. There is a symplectic embedding
\[
\coprod_i\op{int}B^4(a_i) \longrightarrow X_\Omega
\]
which fills the volume of $X_\Omega$. It follows from the Monotonicity property that
\[
c_k(X_\Omega) \ge c_k\left(\coprod_{i\le k}B^4(a_i)\right).
\]
By the Ball and Disjoint Union properties, we have
\[
 c_k\left(\coprod_{i\le k}B^4(a_i)\right) = c_k^{\op{ECH}}\left(\coprod_{i\le k}B^4(a_i)\right).
\]
It is shown in \cite[\S2]{concave} by a combinatorial calculation that
\[
 c_k^{\op{ECH}}\left(\coprod_{i\le k}B^4(a_i)\right) \ge  \max\{\ell_\Omega(\Lambda) \mid \check{\mc{L}}(\Lambda)=k\}.
\]
By Remark~\ref{rem:ECHconcave} and Theorem~\ref{thm:compare}, the above inequalities are equalities.
\end{proof}

We now consider some closed symplectic manifolds. Given $a>0$, let $\C P^2(a)$ denote $\C P^2$ with the Fubini-Study form, scaled so that a line has symplectic area $a$. Let $S^2(a)$ denote $S^2$ with a symplectic form of area $a$.

\begin{theorem}
\label{thm:closed}
Let $a,b>0$ and let $k$ be a nonnegative integer.
\begin{description}
\item{(a)}
$c_k(\C P^2(a))=da$ where $d$ is the unique nonnegative integer with $d^2+d\le 2k \le d^2+3d$.
\item{(b)}
$c_k(S^2(a)\times S^2(b)) = \min\{am+bn \mid m,n\in\Z_{\ge 0}, \; (m+1)(n+1)\ge k+1\}$.
\end{description}
\end{theorem}

To prepare for the proof of this theorem, if $(X,\omega)$ is a closed symplectic four-manifold with $b_2^+(X)=1$, and if $A\in H_2(X)$, let $SW(X,\omega,A)\in\Z/2$ denote the mod $2$ Seiberg-Witten invariant of $X$, for the spin-c structure determined by $\omega$ and $A$, in the symplectic chamber; see the review in \cite[\S2]{cw}. Define the ECH index
\[
I(A) = A\cdot A + \langle c_1(TX), A\rangle \in \Z.
\]

\begin{lemma}
\label{lem:swgr}
Let $(X,\omega)$ be a closed symplectic four-manifold\footnote{If $b_2^+(X)>1$ then the lemma is also true (now the Seiberg-Witten invariant does not depend on a choice of chamber), but vacuous, because in this case one of the corollaries of Taubes's ``Seiberg-Witten = Gromov'' theorem in \cite{swgr} is that $SW(X,\omega,A)\neq 0$ implies $I(A)=0$.} with $b_2^+(X)=1$ and let $A\in H_2(X)$.
If $SW(X,\omega,A)\neq 0$ and $I(A)=2k$, then $c_k(X,\omega) \le \langle[\omega],A\rangle$.
\end{lemma}

\begin{proof}
If $J\in\mc{J}(X,\omega)$ and $x_1,\ldots,x_k\in X$ are generic, then it follows from Taubes's ``Seiberg-Witten = Gromov'' theorem \cite{swgr} that there exists a $J$-holomorphic curve (possibly with disconnected domain) in the homology class $A$ passing through the points $x_1,\ldots,x_k$. Thus
\[
\inf_{u\in \mc{M}^J(X;x_1,\ldots,x_k)}\mc{E}(u) \le \langle [\omega],A\rangle
\]
when $J,x_1,\ldots,x_k$ are generic. A Gromov compactness argument shows that the supremum in the definition \eqref{eqn:maindef} of $c_k(X,\omega)$ is realized for generic $J,x_1,\ldots,x_k$.
\end{proof}

\begin{proof}[Proof of Theorem~\ref{thm:closed}.]
(a) Let $d$ be the integer in the statement of the theorem. Then by \eqref{eqn:balllb} and the Conformality, Monotonicity, and Increasing properties, we have $c_k(\C P^2(a)) \ge da$. On the other hand, by \eqref{eqn:fs} and the Conformality, Monotonicity, and Increasing properties, we have $c_k(\C P^2(a)) \le da$. The latter inequality also follows from Lemma~\ref{lem:swgr} and the Increasing property, because if $A\in H_2(\C P^2)$ is $d$ times the homology class of a line, then $I(A) = d^2+3d$, and as reviewed in \cite[\S2.4]{cw} we have $SW(A)\neq 0$.

(b) Let $L$ denote the right hand side of the equation in (b). If $m$ and $n$ are nonnegative integers, and if $A=(m,n)\in H_2(S^2\times S^2)$, then $I(A) = 2(mn+m+n)$. As reviewed in \cite[\S2.4]{cw}, we have $SW(A)\neq 0$. It follows from Lemma~\ref{lem:swgr} and the Increasing property that
\[
c_k(S^2(a)\times S^2(b)) \le L.
\]

To prove the reverse inequality, consider the polydisk
\begin{equation}
\label{eqn:polydisk}
P(a,b) = \left\{z\in\C^2 \;\big|\; \pi|z_1|^2\le a, \; \pi|z_2|^2\le b\right\}.
\end{equation}
A calculation using Theorem~\ref{thm:convex} shows that
\[
c_k(P(a,b)) = L.
\]
Since the interior of $P(a,b)$ symplectically embeds into $S^2(a)\times S^2(b)$, we are done by Monotonicity.
\end{proof}

\begin{remark}
\label{rem:closed}
Theorem~\ref{thm:closed} shows that the capacities $c_k$ are the same for $\C P^2(a)$ and the ball $B^4(a)$; and likewise they are the same for $S^2(a)\times S^2(b)$ and the polydisk $P(a,b)$. This means that if the capacities $c_k$ obstruct a symplectic embedding of a symplectic four-manifold $(X,\omega)$ into $B^4(a)$ or $P(a,b)$ respectively, then a symplectic embedding of $(X,\omega)$ into $\C P^2(a)$ or $S^2(a)\times S^2(b)$ respectively is not possible either. The same statement is true for the ECH capacities $c_k^{\op{ECH}}$ when $X$ is a star-shaped domain by \cite[Thm.\ 1.4]{cw}.
\end{remark}


\section{An even simpler definition of capacities}
\label{sec:simple}

To conclude, we now define an even simpler series of symplectic capacities, for symplectic manifolds of any dimension.

If $(X,\omega)$ is a symplectic manifold, let $\mc{J}(X,\omega)$ denote the set of $\omega$-compatible almost complex structures on $X$. Given $J\in\mc{J}(X,\omega)$, let $\mc{P}^J(X)$ denote the set of proper holomorphic maps
\[
u:(S,j) \longrightarrow (X,J)
\]
where $(S,j)$ is a one-dimensional complex manifold (not necessarily compact or connected), and we assume that the restriction of $u$ to each component of $S$ is nonconstant. Note that regarded as a two-dimensional real manifold, $S$ does not have boundary. Given $u$ as above, define the {\em energy\/}
\[
\mc{E}(u) = \int_Su^*\omega \in [0,\infty].
\]
Note that the energy is well-defined because $u^*\omega$ is pointwise nonnegative. If $x_1,\ldots,x_k\in X$ are distinct, let $\mc{P}^J(X;x_1,\ldots,x_k)$ denote the set of proper maps $u$ as above such that $x_1,\ldots,x_k\in u(S)$.

\begin{definition}
\label{def:simple}
Let $(X,\omega)$ be a compact symplectic manifold (possibly disconnected and/or with boundary), and let $k$ be a nonnegative integer. Define
\begin{equation}
\label{eqn:chat}
\widehat{c}_k(X,\omega) = \sup_{\substack{J\in\mc{J}(X,\omega)\\ \mbox{\scriptsize $x_1,\ldots,x_k\in \op{int}(X)$ distinct}}} \inf_{u\in\mc{P}^J(\op{int}(X);x_1,\ldots,x_k)} \mc{E}(u) \in [0,\infty].
\end{equation}
\end{definition}

\begin{remark}
It follows immediately from the definition that the capacities $\widehat{c}_k$ satisfy the Conformality, Increasing, Disjoint Union, and Sublinearity properties in Theorem~\ref{thm:properties}.

We can also quickly show that they satisfy Monotonicity under symplectic embeddings $\varphi:(X,\omega)\to(X',\omega')$ between symplectic manifolds of the same dimension, without using Gromov compactness. This is because since $X$ is compact, any $J\in\mc{J}(X,\omega)$ can be extended to $J'\in\mc{J}(X',\omega')$ with $J'|_{\varphi(X)} = \varphi_*J$.

One can now further deduce that each $\widehat{c}_k$ is a $C^0$-continuous function on the set of star-shaped domains in $\R^{2n}$.
\end{remark}

\begin{remark}
When $k=1$, the capacity $\widehat{c}_1$ is very similar\footnote{The only difference is that Gromov uses tame rather than compatible almost complex structures.} to the ``symplectic width'' defined by Gromov in \cite[\S4.1]{gromov}. In particular, $\widehat{c}_1(B^{2n}(a))=a$. The symplectic width should not be confused with the Gromov width $c_{\op{Gr}}$ from \S\ref{sec:intro}. The Monotonicity property of $\widehat{c}_1$ implies that $c_{\op{Gr}} \le \widehat{c}_1$.
\end{remark}

In a sense the capacities $\widehat{c}_k$ are more natural than the $c_k$, because for domains that are not Liouville domains, they are defined directly, without taking a supremum over symplectic embeddings as in Definition~\ref{def:sup}. However the price for this is that we have to consider holomorphic curves without nice boundary conditions, which makes computations more difficult.

\begin{remark}
Suppose that $\op{dim}(X)=4$. If $X$ is closed, then $\widehat{c}_k(X,\omega) = c_k(X,\omega)$ by definition. If $(X,\omega)$ is a Liouville domain, then we have
\begin{equation}
\label{eqn:chatcompare}
\widehat{c}_k(X,\omega) \le c_k(X,\omega).
\end{equation}
This is because if $\epsilon>0$, then any almost complex structure $J\in\mc{J}(X,\omega)$ can be extended to an $\overline{\omega}$-compatible almost complex structure on $\overline{X}$ whose restriction to $[\epsilon,\infty)\times Y$ agrees with an $e^\epsilon\lambda$-compatible almost complex structure on $\R\times Y$. It follows from this as in \eqref{eqn:shrinkingcapacity} that
\[
\widehat{c}_k(X,\omega) \le e^\epsilon c_k(X,\omega).
\]
We can choose $\epsilon>0$ arbitrarily small, and this proves \eqref{eqn:chatcompare}.

We conjecture that in fact $\widehat{c}_k(X,\omega) = c_k(X,\omega)$ when $(X,\omega)$ is a four-dimensional Liouville domain.
\end{remark}

\begin{example}
The simplest example of $\widehat{c}_k$ that we do not know how to compute is $\widehat{c}_3$ of a four-dimensional ball. We currently just know that
\[
\frac{3}{2} \le \widehat{c}_3(B^4(1)) \le 2.
\]
Here the first inequality holds because three copies of $\op{int}(B^4(1/2))$ can be symplectically embedded into $B^4(1)$, and the second inequality holds because $\op{int}(B^4(1))$ can be symplectically embedded into $\C P^2(1)$.
\end{example}

A number of additional capacities along the lines of the $\widehat{c}_k$, defined using proper holomorphic maps satisfying various constraints, are studied in \cite{mikhalkin}.


\end{document}